\documentclass[11pt, a4paper]{article}

\usepackage{lineno}

\usepackage{graphicx}
\usepackage{ae}
\usepackage{amsmath}
\usepackage{amssymb}
\usepackage{latexsym}
\usepackage{url}
\usepackage{epsfig}
\usepackage{cite}
\usepackage{mathrsfs}
\usepackage{amsfonts}
\usepackage{amsthm}

\def\qed{\hfill$\Box$\vspace{11pt}}

\usepackage{color}

\input amssym.def
\input amssym.tex

\long\def\delete#1{}

\newcommand{\be}{\begin{equation}}
\newcommand{\ee}{\end{equation}}
\newcommand{\bea}{\begin{eqnarray}}
\newcommand{\eea}{\end{eqnarray}}
\newcommand{\bean}{\begin{eqnarray*}}
\newcommand{\eean}{\end{eqnarray*}}

\def\Cay{{\rm Cay}}

\def\Spec{{\rm Spec}}
\def\mod{{\rm mod}}

\newtheorem{thm}{Theorem}[section]
\newtheorem{cor}[thm]{Corollary}
\newtheorem{lem}[thm]{Lemma}

\newtheorem{defn}[thm]{Definition}
\newtheorem{assump}[thm]{Assumption}
\newtheorem{rem}{Remark}
\numberwithin{equation}{section}

\title{Quadratic unitary Cayley graphs of finite commutative rings}

\author{Xiaogang Liu$^{1,2}$\, and\, Sanming Zhou$^{1}$\\
\small 1. School of Mathematics and Statistics\\[-0.8ex]
\small The University of Melbourne\\[-0.8ex]
\small Parkville, VIC 3010, Australia
\and
\small 2. Department of Applied Mathematics\\[-0.8ex]
\small Northwestern Polytechnical University\\[-0.8ex]
\small Xi'an, Shaanxi 710072, PR China\\
\small \tt xiaogliu.yzhang@gmail.com, smzhou@ms.unimelb.edu.au}

\date{}

\begin{document}

\openup 0.5\jot
\maketitle

\begin{abstract}
The purpose of this paper is to study spectral properties of a family of Cayley graphs on finite commutative rings. Let $R$ be such a ring and $R^\times$ its set of units. Let $Q_R=\{u^2: u\in R^\times\}$ and $T_R=Q_R\cup(-Q_R)$. We define the quadratic unitary Cayley graph of $R$, denoted by $\mathcal{G}_R$, to be the Cayley graph on the additive group of $R$ with respect to $T_R$; that is, $\mathcal{G}_R$ has vertex set $R$ such that $x, y \in R$ are adjacent if and only if $x-y\in T_R$. It is well known that any finite commutative ring $R$ can be decomposed as $R=R_1\times R_2\times\cdots\times R_s$, where each $R_i$ is a local ring with maximal ideal $M_i$. Let $R_0$ be a local ring with maximal ideal $M_0$ such that $|R_0|/|M_0| \equiv 3\,(\mod\,4)$. We determine the spectra of $\mathcal{G}_R$ and $\mathcal{G}_{R_0\times R}$ under the condition that $|R_i|/|M_i|\equiv 1\,(\mod\,4)$ for $1 \le i \le s$. We compute the energies and spectral moments of such quadratic unitary Cayley graphs, and determine when such a graph is hyperenergetic or Ramanujan.

\bigskip


\bigskip

\end{abstract}

\section{Introduction}

The {\em adjacency matrix} of a graph is the matrix with rows and columns indexed by its vertices such that the $(i,j)$-entry is equal to $1$ if vertices $i$ and $j$ are adjacent and $0$ otherwise. The \emph{eigenvalues} of a graph are eigenvalues of its adjacency matrix, and the \emph{spectrum} of a graph is the collection of its eigenvalues with multiplicities. If $\lambda_1,\lambda_2,\ldots,\lambda_k$ are distinct eigenvalues of a graph $G$ and $m_1,m_2,\ldots,m_k$ the corresponding multiplicities, then we denote the spectrum of $G$ by
\begin{eqnarray*}
\Spec (G)=\left(\begin{array}{ccc}
\lambda_1 & \ldots  & \lambda_k \\
m_1      & \ldots  & m_k
\end{array}
\right),
\end{eqnarray*}
or simply by $\lambda_1^{m_1},\ldots,\lambda_k^{m_k}$ with $m_i$ omitted if it is equal to $1$. The \emph{energy} $E(G)$ of $G$ is defined as the sum of the absolute values of its eigenvalues; that is, $E(G) = \sum_{i=1}^k m_i |\lambda_i|$. This notion was introduced by I. Gutman \cite{kn:Gutman78} in the context of mathematical chemistry. It is a graph parameter arising from the H\"{u}ckel  molecular orbital approximation for the total $\pi$-electron energy. The energies of graphs have been studied extensively in recent years \cite{kn:Brualdiw, kn:Gutman01, Gutman99, kn:Gutman10, kn:Ilic09, kn:Ilic11, kn:Kiani11, kn:Kiani12, kn:Li12, kn:LiuZhou12, kn:Ramaswamy09, kn:Rojo11, kn:Rojo111, kn:Sander111, kn:Sander11}.

A finite $r$-regular graph $G$ is called \emph{Ramanujan} \cite{HLW, kn:Murty03} if
$\lambda(G) \leq 2\sqrt{r-1},$
where $\lambda(G)$ is the maximum in absolute value of an eigenvalue of $G$ other than $\pm r$. This notion arises from the well known Alon-Boppana bound (see e.g.~\cite[Theorem 0.8.8]{DSV}), which asserts that $\liminf_{i \rightarrow \infty} \lambda(G_i) \ge 2\sqrt{r-1}$ for any family of finite, connected, $r$-regular graphs $\{G_i\}_{i \ge 1}$ with $|V(G_i)| \rightarrow \infty$ as $i \rightarrow \infty$. Over many years a significant amount of work has been done on Ramanujan graphs \cite{DSV, kn:Murty03} and related expander graphs \cite{HLW} with an emphasis on explicit constructions.
The \emph{$k$-th spectral moment} of a graph $G$ with (not necessarily distinct) eigenvalues $\lambda_1,\lambda_2,\ldots,\lambda_n$ is defined as
$s_k(G)=\sum_{i=1}^n \lambda_i^k,$
where $n$ is the number of vertices of $G$ and $k \ge 0$ is an integer. Spectral moments are related to many combinatorial properties of graphs. For example, they play an important role in the proof by Lubotzky, Phillips and Sarnak \cite{LPS} of the Alon-Boppana bound, and the 4th spectral moment was used in \cite{RT} to give an upper bound on the energy of a bipartite graph.

The purpose of this paper is to study spectral properties of a family of Cayley graphs on finite commutative rings. A \emph{local ring}  \cite{kn:Atiyah69} is a commutative ring with a unique maximal ideal. It is readily seen \cite{kn:Atiyah69,kn:Dummit03} that the set of units of a local ring $R$ with maximal ideal $M$ is given by $R^\times=R\setminus M$. It is well known \cite{kn:Atiyah69,kn:Dummit03} that every finite commutative ring can be expressed as a direct product of finite local rings, and this decomposition is unique up to permutations of such local rings.
\begin{assump}
\label{as:1}
Whenever we consider a finite commutative ring $R=R_1\times R_2\times\cdots\times R_s$ with unit element $1 \neq 0$, we assume that each $R_i$, $1 \le i \le s$, is a local ring with maximal ideal $M_i$ of order $m_i$ such that
\[|R_1|/m_1 \leq |R_2|/m_2 \leq \cdots \leq |R_s|/m_s.\]
\end{assump}
Denote by $R^\times$ the set of units of $R$. Then
\begin{equation}
\label{eq:basic}
|R^\times|=\prod_{i=1}^s(|R_i|-m_i)=\prod_{i=1}^s m_i \left((|R_{i}|/m_{i})-1\right)=|R|\prod_{i=1}^s\left(1-\frac{1}{|R_i|/m_i}\right).
\end{equation}
We study the following family of Cayley graphs with a focus on their spectral properties.

\begin{defn}
{\em Given a finite commutative ring $R$, the \emph{quadratic unitary Cayley graph} of $R$, denoted by $\mathcal{G}_R$, is defined as the Cayley graph $\Cay(R,T_R)$ on the additive group of $R$ with respect to $T_R=Q_R\cup(-Q_R)$, where $Q_R=\{u^2: u\in R^\times\}$. That is, $\mathcal{G}_R$ has vertex set $R$ such that $x, y \in R$ are adjacent if and only if $x-y\in T_R$.}
\end{defn}

This notion is a generalization of the quadratic unitary Cayley graph $\mathcal{G}_{\mathbb{Z}_n}$ of $\mathbb{Z}_n$ introduced in \cite{kn:Beaudrap10}. With a focus on structural properties of $\mathcal{G}_{\mathbb{Z}_n}$, de Beaudrap \cite{kn:Beaudrap10} characterized decompositions of $\mathcal{G}_{\mathbb{Z}_n}$ into tensor products over relatively prime factors of $n$. He also computed the diameter of $\mathcal{G}_{\mathbb{Z}_n}$ and gave conditions under which $\mathcal{G}_{\mathbb{Z}_n}$ is perfect.

Quadratic unitary Cayley graphs are also generalizations of the well-known Paley graphs. In fact, in the special case where $R = \mathbb{F}_q$ is a finite field, where $q \equiv 1\,(\mod\,4)$ is a prime power, $\mathcal{G}_{\mathbb{F}_q}$ is exactly the Paley graph $P(q)$, which by definition is the graph with vertex set $\mathbb{F}_q$ such that $x, y \in \mathbb{F}_q$ are adjacent if and only if $x-y$ is a non-zero square of $\mathbb{F}_q$.

The main results of this paper are as follows.
Let $R$ be as in Assumption \ref{as:1} such that $|R_i|/m_i\equiv 1\,(\mod\,4)$ for $1 \le i \le s$, and $R_0$ a local ring with maximal ideal $M_0$ of order $m_0$ such that $|R_0|/m_0\equiv 3\,(\mod\,4)$. We first compute the spectra of $\mathcal{G}_R$ and $\mathcal{G}_{R_0\times R}$ (see Theorems \ref{SPECLocalQUCG}, \ref{Spec1mod4} and \ref{Spec3mod4}). By using these spectra, we determine the energies (Theorems \ref{EnergyLQUCG} and \ref{Energy13mod4}) and spectral moments (Theorems \ref{SMLQUCG} and \ref{SMs13mod4}) of such quadratic unitary Cayley graphs and find out when such a graph is hyperenergetic (Corollary \ref{HyEnerCor1}) or Ramanujan (Theorem \ref{RamQUCG13mod4}). Corresponding results for $\mathcal{G}_{\mathbb{Z}_n}$ will be given as corollaries.

\section{Spectra of quadratic unitary Cayley graphs}
\label{sec:SpecQUCA}

We first recall the following well known result.

\begin{lem}\label{localmaximal}\emph{\cite[Proposition 2.1]{kn:Akhtar09}}
Let $R$ be a finite local ring and $m$ the order of its unique maximal ideal. Then there exists a prime $p$ such that $|R|$, $m$ and $|R|/m$ are all powers of $p$.
\end{lem}

A pseudograph is obtained from a graph by adding a loop at some of the vertices. The \emph{tensor product} of two graphs (or pseudographs) $G$ and $H$, denoted by $G\otimes H$, is the graph with vertex set $V(G)\times V(H)$ in which $(u,v)$ is adjacent to $(x,y)$ if and only if $u$ is adjacent to $x$ in $G$ and $v$ is adjacent to $y$ in $H$. It is known \cite{HIK} that this operation is associative and so the tensor product $G_1 \otimes G_2 \otimes \cdots \otimes G_s$ of any given (pseudo) graphs $G_1, G_2, \ldots, G_s$ (where $s \ge 1$) is well-defined.

\begin{lem}
\label{tensorspectrum}
\emph{\cite[Theorem 2.5.4]{kn:Cvetkovic10}}
Let $G$ and $H$ be graphs (or pseudographs) with eigenvalues $\lambda_1,\lambda_2,\ldots,\lambda_n$ and $\mu_1,\mu_2,\ldots,\mu_m$, respectively. Then the eigenvalues of $G\otimes H$ are $\lambda_i\mu_j$, $1\leq i\leq n, 1\leq j\leq m$.
\end{lem}

Let $K_n$ denote the complete graph on $n$ vertices and $\mathring{K}_n$ the \emph{complete pseudograph} obtained by attaching a loop to each vertex of $K_n$. It can be verified that $G \otimes \mathring{K}_n$ is simply the lexicographic product of $G$ by the empty graph of $n$ vertices.

\begin{thm}\label{LocalQUCG}
Let $R$ be a finite local ring with maximal ideal $M$. If $|R|/|M|$ is odd, then $\mathcal{G}_{R}\cong\mathcal{G}_{R/M}\otimes\mathring{K}_{|M|}$.
\end{thm}
\begin{proof}
Since $R$ is a local ring and $M$ is its maximal ideal, we have $R^\times = R \setminus M$ and $R/M$ is a finite field. By Lemma \ref{localmaximal} we have $|R|/|M|=p^{s}$ and $|M| = p^t$ for a prime $p$ and some integers $s \ge1, t \ge 0$. Define $\rho:R^\times\rightarrow(R/M)^\times$ by $\rho(r)=r+M$ for $r\in R^\times$. Then $\rho$ is a well-defined surjective homomorphism from the multiplicative group $R^\times$ to the multiplicative group $(R/M)^\times$ with kernel $\ker(\rho) = 1+M$. Thus $R^\times/(1+M)\cong(R/M)^\times$ and the corresponding isomorphism from $R^\times/(1+M)$ to $(R/M)^\times$ is given by $r(1+M) \mapsto r+M$ for $r\in R^\times$. Since $|1+M|=p^t$ and $|R^\times|/|1+M| = (|R|-|M|)/|M|=p^s-1$, $|1+M|$ and $|R^\times|/|1+M|$ are coprime, and so $1+M$ is a Sylow $p$-subgroup of $R^\times$. Thus, $R^\times \cong R^\times/(1+M) \times (1+M)$, say, with isomorphism given by $r \mapsto (\hat{r}(1+M), 1+m_r)$, where $\hat{r} \in R$ and $m_r \in M$ are determined by $r \in R^\times$. Since $R^\times/(1+M)\cong(R/M)^\times$, it follows that $R^\times\cong(R/M)^\times\times(1+M)$ and the corresponding isomorphism is given by $\psi(r) = (\hat{r} + M, 1+m_r)$ for $r \in R^\times$. Since $1+M$ is a Sylow $p$-subgroup of $R^\times$ and $p$ is odd, we have $(1+M)^2 = 1+M$. This together with $R^\times\cong(R/M)^\times\times(1+M)$ implies that $Q_R\cong Q_{R/M}\times(1+M)$ as groups with the corresponding isomorphism giving by $\psi(r^2) = (\hat{r}^2 + M, (1+m_r)^2)$ for $r \in R^\times$. (Note that $Q_{R/M}$ consists of nonzero squares of $R/M$ since $R/M$ is a field.) Write $R/M=\{r_1+M,r_2+M,\ldots,r_{p^{s}}+M\}$. Then for each $r\in R$ there is a unique $i$ and $n_r\in M$ such that $r=r_i+n_r$. Let $\tau:R\rightarrow R/M\times M$ be defined by $\tau(r)=(r_i+M,n_r)=(r+M,n_r)$. Since $\tau$ is clearly surjective, it is a bijection from $R$ to $R/M\times M$ as the two sets have the same size.

We are now ready to prove $\mathcal{G}_{R}\cong\mathcal{G}_{R/M}\otimes\mathring{K}_{|M|}$. We treat $\mathcal{G}_{R/M}\otimes\mathring{K}_{|M|}$ as defined on $R/M \times M$ such that $(x+M, a), (y+M, b)$ are adjacent if and only if $x+M, y+M$ are adjacent in $\mathcal{G}_{R/M}$ (that is, $(x-y)+M$ or $(y-x)+M$ belongs to $Q_{R/M}$). Suppose that $x, y \in R$ are adjacent in $\mathcal{G}_{R}$, that is, $x-y = \pm r^2$ for some $r \in R^\times$. Without loss of generality we may assume $x-y = r^2$ so that $(x+M) - (y+M) = r^2 + M = (r+M)^2$. Since $r \not \in M$ and $M$ is the zero-element of the field $R/M$, it follows that $(x+M) - (y+M) \in Q_{R/M}$. Therefore, $\tau(x)$ and $\tau(y)$ are adjacent in $\mathcal{G}_{R/M}\otimes\mathring{K}_{|M|}$. So we have proved that $\mathcal{G}_{R}$ is embedded into $\mathcal{G}_{R/M}\otimes\mathring{K}_{|M|}$ via $\tau$ as a spanning subgraph since the two graphs have the same number of vertices. On the other hand, if $|R|/m\equiv 1\,(\mod\,4)$, then $-1 \in Q_{R/M}$ and $-1 \in Q_{R}$, and the degree of $\mathcal{G}_{R/M}\otimes\mathring{K}_{|M|}$ is equal to $|Q_{R/M}| |M| = |Q_R|$, which is the same as the degree of $\mathcal{G}_{R}$. If $|R|/m\equiv 3\,(\mod\,4)$, then $-1 \notin Q_{R/M}$ and $-1 \notin Q_{R}$, and the degree of $\mathcal{G}_{R/M}\otimes\mathring{K}_{|M|}$ is equal to $2 |Q_{R/M}| |M| = 2 |Q_R|$, which is also the same as the degree of $\mathcal{G}_{R}$. In either case $\mathcal{G}_{R}$ and $\mathcal{G}_{R/M}\otimes\mathring{K}_{|M|}$ must be isomorphic to each other because they have the same degree and one is a spanning subgraph of the other.
\qed
\end{proof}

\begin{thm}\label{SPECLocalQUCG}
Let $R$ be a local ring with maximal ideal $M$ of order $m$.
\begin{itemize}
\item[\rm (a)] If $|R|/m\equiv 1\,(\mod\,4)$, then
\begin{eqnarray*}
\Spec (\mathcal{G}_{R})=\left(\begin{array}{cccc}
\dfrac{|R|-m}{2} & \dfrac{m\left(-1+\sqrt{|R|/m}\right)}{2}  & \dfrac{m\left(-1-\sqrt{|R|/m}\right)}{2} & 0    \\[0.3cm]
1                    &  (|R|/m-1)/2                          &  (|R|/m-1)/2                             & |R|-|R|/m
\end{array}
\right);
\end{eqnarray*}
\item[\rm (b)] if $|R|/m\equiv 3\,(\mod\,4)$, then \begin{eqnarray*}
\Spec (\mathcal{G}_{R})=\left(\begin{array}{ccc}
|R|-m   & -m              &  0       \\[0.1cm]
1       & |R|/m-1         & |R|-|R|/m
\end{array}
\right).
\end{eqnarray*}
\end{itemize}
\end{thm}
\begin{proof}
(a) If $|R|/m\equiv 1\,(\mod\,4)$, then $|R|/m$ is an odd prime power and $\mathcal{G}_{R/M}$ coincides with the Paley graph of order $|R|/m$. The latter has spectrum \cite{kn:Brouwer12}
$$
\frac{|R|/m-1}{2},\; \left(\frac{-1+\sqrt{|R|/m}}{2}\right)^{(|R|/m-1)/2},\; \left(\frac{-1-\sqrt{|R|/m}}{2}\right)^{(|R|/m-1)/2}.
$$
Combining this with Lemma \ref{tensorspectrum}, Theorem \ref{LocalQUCG} and the fact that the spectrum of $\mathring{K}_{m}$ is $m$, $0^{m-1}$, we obtain the required result.

(b) If $|R|/m\equiv 3\,(\mod\,4)$, then $|R|/m$ is an odd prime power and $\mathcal{G}_{R/M}$ is a complete graph of order $|R|/m$. Since $K_{|R|/m}$ has spectrum $|R|/m-1$, $(-1)^{|R|/m-1}$ \cite{kn:Brouwer12}, the required result is obtained by applying Lemma \ref{tensorspectrum} and Theorem \ref{LocalQUCG}.
\qed\end{proof}

The following result is a generalization of \cite[Theorem 1]{kn:Beaudrap10}, and its proof is similar to the proof of \cite[Theorem 1]{kn:Beaudrap10}. (Note that we take $\mathcal{G}_R$ as defined on the Cartesian product $R_1\times R_2\times\cdots\times R_s$. By abuse notation, in the sequel we denote by $-1$ the negative element of the unit $1$ in different rings but it should be easy to identify the ring involved.)

\begin{thm}\label{CProdQUCG}
Let $R$ be as in Assumption \ref{as:1}. Then $\mathcal{G}_R =  \mathcal{G}_{R_1}\otimes\mathcal{G}_{R_2}\otimes\cdots\otimes\mathcal{G}_{R_s}$ if and only if there exists at most one $R_j$ such that $-1 \notin Q_{R_j}$.
\end{thm}

\begin{proof}
Since the result is trivial when $s=1$, we assume $s \ge 2$.

Consider the case $s=2$ first. In this case we are required to prove that $\mathcal{G}_R =  \mathcal{G}_{R_1}\otimes\mathcal{G}_{R_2}$ if and only if $-1 \in Q_{R_1}$ or $-1\in Q_{R_2}$. In fact, we prove the following stronger result: For any finite commutative rings $A$ and $B$,
\begin{equation}
\label{eq:2}
\mathcal{G}_{A \times B} = \mathcal{G}_{A}\otimes\mathcal{G}_{B} \Leftrightarrow -1 \in Q_{A}\; \mbox{ or}\; -1\in Q_{B}.
\end{equation}
We first notice that
$$
T_{A}\times T_{B} = \{(\pm a^2, \pm b^2), (\pm a^2, \mp b^2): a \in A^{\times}, b \in B^{\times}\}
$$
and
$$
T_{A \times B} = (Q_{A}\times Q_{B})\cup(-Q_{A}\times Q_{B}) = \{\pm (a^2, b^2): a \in A^{\times}, b \in B^{\times}\} \subseteq T_{A}\times T_{B}.
$$
Thus from the definitions of $\mathcal{G}_{A}$, $\mathcal{G}_{B}$ and $\mathcal{G}_{A}\otimes\mathcal{G}_{B}$ it follows that $\mathcal{G}_{A}\otimes\mathcal{G}_{B} = \Cay(R, T_{A}\times T_{B})$. Therefore, $\mathcal{G}_{A \times B} = \mathcal{G}_{A}\otimes\mathcal{G}_{B}$ if and only if $T_{A \times B} =  T_{A}\times T_{B}$.

If $-1\notin Q_{A}$ and $-1\notin Q_{B}$, then both $(-1,1)$ and $(1,-1)$ are elements of $T_{A}\times T_{B}$ but neither of them is an element of $T_{A \times B}$. Thus $T_{A \times B} \neq T_{A}\times T_{B}$ and so $\mathcal{G}_{A \times B} \neq  \mathcal{G}_{A}\otimes\mathcal{G}_{B}$.
On the other hand, suppose that at least one of $Q_{A}$ and $Q_{B}$ contains $-1$. Without loss of generality we may suppose $-1\in Q_{A}$ so that $i^2=-1$ for some $i\in A^\times$. Then for any $(a, b) \in (A \times B)^\times$ and $s, t \in\{0,1\}$ we have
$$
\left((-1)^{s}a^2, (-1)^{t} b^2\right) = (-1)^{t}\left((-1)^{(s-t)}a^2, b^2\right) = (-1)^{t}\left((i^{(s-t)}a)^2, b^2\right).
$$
Hence $T_{A}\times T_{B}\subseteq T_{A \times B}$. Therefore, $T_{A \times B} =  T_{A}\times T_{B}$ and $\mathcal{G}_{A \times B} = \mathcal{G}_{A}\otimes\mathcal{G}_{B}$. Thus (\ref{eq:2}) is proved and so the result in the theorem is true when $s=2$.

In general, we make induction on $s$ based on (\ref{eq:2}). Suppose that for some integer $k\ge2$ the result holds when $s \le k$, that is, $\mathcal{G}_{R_1\times R_2\times \cdots\times R_s} = \mathcal{G}_{R_1}\otimes \mathcal{G}_{R_2}\otimes \cdots\otimes\mathcal{G}_{R_s}$ if and only if there exists at most one $j$ between $1$ and $s$ such that $-1 \notin Q_{R_j}$. Based on this hypothesis we aim to prove that the result holds when $s = k+1$. In the rest of the proof we set $R = R_1\times R_2\times \cdots \times R_{k+1}$.

Assume that there is at most one $j$ between $1$ and $k+1$ such that $-1 \notin Q_{R_j}$. Then $(-1, -1, \ldots, -1) \in Q_{R_1\times R_2\times \cdots\times R_k}$ or $-1 \in Q_{R_{k+1}}$ no matter whether $1 \le j \le k$ or $j=k+1$, and hence $\mathcal{G}_R = \mathcal{G}_{R_1\times R_2\times \cdots\times R_k}\otimes\mathcal{G}_{R_{k+1}}$ by (\ref{eq:2}). On the other hand, by the induction hypothesis we have $\mathcal{G}_{R_1\times R_2\times \cdots\times R_k} = \mathcal{G}_{R_1}\otimes \mathcal{G}_{R_2}\otimes \cdots\otimes\mathcal{G}_{R_k}$. Therefore, $\mathcal{G}_R = \mathcal{G}_{R_1} \otimes \mathcal{G}_{R_2} \otimes \cdots\otimes\mathcal{G}_{R_k} \otimes\mathcal{G}_{R_{k+1}}$.

Conversely, assume that $\mathcal{G}_R = \mathcal{G}_{R_1} \otimes \mathcal{G}_{R_2} \otimes \cdots\otimes\mathcal{G}_{R_k} \otimes \mathcal{G}_{R_{k+1}}$. We aim to prove that there exists at most one $j$ between $1$ and $k+1$ such that $-1 \notin Q_{R_j}$. Suppose otherwise. Without loss of generality we may assume that, for some integer $t$ with $2 \le t \le k+1$, we have $-1 \notin Q_{R_j}$ for $1 \le j \le t$ and $-1 \in Q_{R_j}$ for $t  < j \le k+1$. If $t < k+1$, then since $-1 \in Q_{R_{k+1}}$ we have $\mathcal{G}_R = \mathcal{G}_{R_1\times R_2\times \cdots\times R_{k}} \otimes \mathcal{G}_{R_{k+1}}$ by (\ref{eq:2}). Similarly, if $t < k$, then $\mathcal{G}_{ R_1\times R_2\times \cdots\times R_{k}} = \mathcal{G}_{R_1\times R_2\times \cdots\times R_{k-1}} \otimes \mathcal{G}_{R_{k}}$ by (\ref{eq:2}), and hence $\mathcal{G}_R = (\mathcal{G}_{R_1\times R_2\times \cdots\times R_{k-1}} \otimes \mathcal{G}_{R_{k}}) \otimes \mathcal{G}_{R_{k+1}} = \mathcal{G}_{ R_1\times R_2\times \cdots\times R_{k-1}} \otimes \mathcal{G}_{R_{k}} \otimes \mathcal{G}_{R_{k+1}}$. Continuing, we obtain $\mathcal{G}_R = \mathcal{G}_{R_1\times R_2\times \cdots\times R_{t}} \otimes \mathcal{G}_{R_{t+1}} \otimes \cdots \otimes \mathcal{G}_{R_{k+1}}$. Comparing this with the assumption $\mathcal{G}_R = \mathcal{G}_{R_1} \otimes \mathcal{G}_{R_2} \otimes \cdots\otimes\mathcal{G}_{R_k} \otimes \mathcal{G}_{R_{k+1}}$, we obtain that $\mathcal{G}_{R_1\times R_2\times \cdots\times R_{t}} = \mathcal{G}_{R_1}\otimes \mathcal{G}_{R_2}\otimes \cdots\otimes\mathcal{G}_{R_t}$. (We used the fact that, if $G \otimes H_1 = G \otimes H_2$ for graphs $G, H_1, H_2$ with $E(G) \ne \emptyset$ and $V(H_1) = V(H_2)$, then $H_1 = H_2$. This is easy to prove, though the more general cancellation law for the tensor product is not true in general  \cite[Section 9.2]{HIK}.) If $2 \le t \le k$, then by $\mathcal{G}_{R_1\times R_2\times \cdots\times R_{t}} = \mathcal{G}_{R_1}\otimes \mathcal{G}_{R_2}\otimes \cdots\otimes\mathcal{G}_{R_t}$ and the induction hypothesis, there exists at most one $j$ between $1$ and $t$ such that $-1 \notin Q_{R_j}$. Since this contradicts the definition of $t$ and the assumption $t \ge 2$, we conclude that $t = k+1$, that is, $-1 \notin Q_{R_j}$ for $1 \le j \le k+1$. Since $T_{R} = (Q_{R_1} \times Q_{R_2} \times \cdots \times Q_{R_{k+1}}) \cup (-Q_{R_1} \times Q_{R_2} \times \cdots \times Q_{R_{k+1}})$, it then follows that $(-1, 1, \ldots, 1) \notin T_{R}$. On the other hand, we have $(-1, 1, \ldots, 1) \in T_{R_1} \times T_{R_2} \times \cdots\times T_{R_{k+1}}$. Therefore, $T_{R} \ne T_{R_1} \times T_{R_2} \times \cdots\times T_{R_{k+1}}$ and consequently $\mathcal{G}_R \ne \Cay(R, T_{R_1} \times T_{R_2} \times \cdots\times T_{R_{k+1}}) = \mathcal{G}_{R_1} \otimes \mathcal{G}_{R_2} \otimes \cdots\otimes\mathcal{G}_{R_k} \otimes \mathcal{G}_{R_{k+1}}$. This contradiction shows that there is at most one $j$ between $1$ and $k+1$ such that $-1 \notin Q_{R_j}$.
\qed\end{proof}

Define
$$
\lambda_{A,B} = (-1)^{|B|}\dfrac{|R^\times|}{2^s\prod_{i\in A}\left(\sqrt{|R_i|/m_i}+1\right) \prod_{j\in B}\left(\sqrt{|R_j|/m_j}-1\right)}
$$
for disjoint subsets $A, B$ of $\{1,2,\ldots,s\}$. In particular, $\lambda_{\emptyset,\emptyset} = |R^\times|/2^s$.

\begin{thm}\label{Spec1mod4}
Let $R$ be as in Assumption \ref{as:1} such that $|R_i|/m_i\equiv 1\,(\mod\,4)$ for $1 \le i \le s$.  Then the eigenvalues of $\mathcal{G}_R$ are
\begin{itemize}
\item[\rm (a)] $\lambda_{A,B}$, repeated $\dfrac{1}{2^{|A|+|B|}}\prod\limits_{k\in A\cup B}(|R_k|/m_k-1)$ times, for all pairs $(A, B)$ of subsets of $\{1,2,\ldots,s\}$ such that $A\cap B=\emptyset$; and
\item[\rm (b)] $0$ with multiplicity $|R|-\sum\limits_{\substack{A, B\subseteq\{1,\ldots,s\}\\ A\cap B=\emptyset}}\left(\dfrac{1}{2^{|A|+|B|}}\prod\limits_{k\in A\cup B}(|R_k|/m_k-1)\right)$.
\end{itemize}
\end{thm}
\begin{proof}
Since $|R_i|/m_i\equiv 1\,(\mod\,4)$ for $1 \le i \le s$, we have $-1 \in Q_{R_i/M_i}$ for $1 \le i \le s$, which implies $-1 \in Q_{R_i}$ for $1 \le i \le s$. Then $\mathcal{G}_R\cong \mathcal{G}_{R_1}\otimes \cdots\otimes \mathcal{G}_{R_s}$ by Theorem \ref{CProdQUCG}. Thus the result is obtained by applying Lemma \ref{tensorspectrum} and (a) of Theorem \ref{SPECLocalQUCG}.
\qed\end{proof}

\begin{thm}\label{Spec3mod4}
Let $R$ be as in Assumption \ref{as:1} such that $|R_i|/m_i\equiv 1\,(\mod\,4)$ for $1 \le i \le s$, and let $R_0$ be a local ring with maximal ideal $M_0$ of order $m_0$ such that $|R_0|/m_0\equiv 3\,(\mod\,4)$.  Then the eigenvalues of $\mathcal{G}_{R_0\times R}$ are
\begin{itemize}
\item[\rm (a)] $|R_0^\times|\cdot\lambda_{A,B}$, repeated $\dfrac{1}{2^{|A|+|B|}}\prod\limits_{k\in A\cup B}(|R_k|/m_k-1)$ times, for all pairs $(A, B)$ of subsets of $\{1,2,\ldots,s\}$ such that $A\cap B=\emptyset$;
\item[\rm (b)] $-\dfrac{|R_0^\times|}{|R_0|/m_0-1}\cdot\lambda_{A,B}$, repeated $\dfrac{|R_0|/m_0-1}{2^{|A|+|B|}}\prod\limits_{k\in A\cup B}(|R_k|/m_k-1)$ times, for all pairs $(A, B)$ of subsets of $\{1,2,\ldots,s\}$ such that $A\cap B=\emptyset$; and
\item[\rm (c)] $0$ with multiplicity $|R|-\sum\limits_{\substack{A,B\subseteq\{1,\ldots,s\}\\ A\cap B=\emptyset}}\left(\dfrac{|R_0|/m_0}{2^{|A|+|B|}}\prod\limits_{k\in A\cup B}(|R_k|/m_k-1)\right)$.
\end{itemize}
\end{thm}
\begin{proof}
The result follows from Lemma \ref{tensorspectrum} and Theorems \ref{SPECLocalQUCG} and \ref{CProdQUCG}.
\qed\end{proof}

We now specify the results above to obtain the eigenvalues of ${\cal G}_{\mathbb{Z}_n}$. Let $n=p_1^{\alpha_1}p_2^{\alpha_2}\cdots p_s^{\alpha_s}$ be an integer in canonical factorisation, where $p_1<p_2<\cdots<p_s$ are primes. It is well known (see e.g.~\cite{kn:Dummit03}) that $\mathbb{Z}/n\mathbb{Z}\cong(\mathbb{Z}/p_1^{\alpha_1}\mathbb{Z})\times(\mathbb{Z}/p_2^{\alpha_2}\mathbb{Z})\times\cdots\times(\mathbb{Z}/p_s^{\alpha_s}\mathbb{Z})$,
where each $\mathbb{Z}/p_i^{\alpha_i}\mathbb{Z}$ is a local ring with unique maximal ideal $(p_i)/(p_i^{\alpha_i})$ of order $p_i^{\alpha_i-1}$. Denote by $\varphi$ the Euler's totient function.

\begin{cor}\label{G2pspec}
$\mathrm{(a)}$ If $p\equiv1\,(\mod\,4)$ is a prime and $\alpha\ge1$ an integer, then
\begin{eqnarray*}
\Spec (\mathcal {G}_{\mathbb{Z}_{p^\alpha}})=\left(\begin{array}{cccc}
p^{\alpha-1}(p-1)/2 & p^{\alpha-1}\left(-1+\sqrt{p}\right)/2 & 0      &  p^{\alpha-1}\left(-1-\sqrt{p}\right)/2 \\[0.1cm]
1              & (p-1)/2                & p^\alpha-p  &  (p-1)/2
\end{array}
\right).
\end{eqnarray*}
$\mathrm{(b)}$ If $p\equiv3\,(\mod\,4)$ is a prime and $\alpha\ge1$ an integer, then
\begin{eqnarray*}
\Spec (\mathcal {G}_{\mathbb{Z}_{p^\alpha}})=\left(\begin{array}{ccc}
p^{\alpha-1}(p-1)   & -p^{\alpha-1}        &  0       \\[0.1cm]
1              & p-1         & p^\alpha-p
\end{array}
\right).
\end{eqnarray*}
\end{cor}

\begin{cor}\label{EigenQUCG}
Let $n=p_1^{\alpha_1}\ldots p_s^{\alpha_s}$ be an integer in canonical factorization such that each $p_i \equiv1\,(\mod\,4)$. Then the eigenvalues of ${\cal G}_{\mathbb{Z}_n}$ are
\begin{itemize}
\item[\rm (a)] $(-1)^{|B|}\cdot\dfrac{\varphi(n)}{2^s\prod_{i\in A}\left(\sqrt{p_i}+1\right)\prod_{j\in B}\left(\sqrt{p_j}-1\right)}$, repeated $\dfrac{1}{2^{|A|+|B|}}\prod\limits_{k\in A\cup B}(p_k-1)$ times, for all pairs $(A, B)$ of subsets of $\{1,2,\ldots,s\}$ such that $A\cap B=\emptyset$; and
\item[\rm (b)] $0$ with multiplicity $n-\sum\limits_{\substack{A,\,B\subseteq\{1,\ldots,s\}\\A\cap B=\emptyset}}\left(\dfrac{1}{2^{|A|+|B|}}\prod\limits_{k\in A\cup B}(p_k-1)\right)$.
\end{itemize}
\end{cor}

\begin{cor}
Let $n=p^\alpha p_1^{\alpha_1}\ldots p_s^{\alpha_s}$ be an integer in canonical factorization such that $p\equiv3\,(\mod\,4)$ and each $p_i \equiv1\,(\mod\,4)$. Then the eigenvalues of ${\cal G}_{\mathbb{Z}_n}$ are
\begin{itemize}
\item[\rm (a)] $(-1)^{|B|}\cdot\dfrac{\varphi(n)}{2^s\prod_{i\in A}\left(\sqrt{p_i}+1\right)\prod_{j\in B}\left(\sqrt{p_j}-1\right)}$, repeated $\dfrac{1}{2^{|A|+|B|}}\prod\limits_{k\in A\cup B}(p_k-1)$ times, for all pairs $(A, B)$ of subsets of $\{1,2,\ldots,s\}$ such that $A\cap B=\emptyset$;
\item[\rm (b)] $(-1)^{|B|+1}\cdot\dfrac{\varphi(n)}{2^s(p-1)\prod_{i\in A}\left(\sqrt{p_i}+1\right)\prod_{j\in B}\left(\sqrt{p_j}-1\right)}$, repeated $\dfrac{p-1}{2^{|A|+|B|}}\prod\limits_{k\in A\cup B}(p_k-1)$ times, for all pairs $(A, B)$ of subsets of $\{1,2,\ldots,s\}$ such that $A\cap B=\emptyset$; and
\item[\rm (c)] $0$ with multiplicity $n-\sum\limits_{\substack{A,\,B\subseteq\{1,\ldots,s\}\\A\cap B=\emptyset}}\left(\dfrac{p}{2^{|A|+|B|}}\prod\limits_{k\in A\cup B}(p_k-1)\right)$.
\end{itemize}
\end{cor}

\section{Energies of quadratic unitary Cayley graphs}
\label{sec:EnergyQUCA}

The following is an immediate consequence of Theorem \ref{SPECLocalQUCG}.

\begin{thm}\label{EnergyLQUCG}
Let $R$ be a local ring with maximal ideal $M$ of order $m$.
\begin{itemize}
\item[\rm (a)] If $|R|/m\equiv 1\,(\mod\,4)$, then $E(\mathcal{G}_{R})=\left(\sqrt{|R|/m}+1\right)|R^\times|\big/2$;
\item[\rm (b)] if $|R|/m\equiv 3\,(\mod\,4)$, then $E(\mathcal{G}_{R})=2|R^\times|$.
\end{itemize}
\end{thm}

As a consequence of Lemma \ref{tensorspectrum}, we have $E(G\otimes H) = E(G)\cdot E(H)$. In general, by induction we see that the energy of the tensor product of a finite number of graphs is equal to the product of the energies of the factor graphs. This together with Theorem \ref{EnergyLQUCG} yields the following result.

\begin{thm}\label{Energy13mod4}
Let $R$ be as in Assumption \ref{as:1} such that $|R_i|/m_i\equiv 1\,(\mod\,4)$ for $1 \le i \le s$, and $R_0$ a local ring with maximal ideal $M_0$ of order $m_0$ such that $|R_0|/m_0\equiv 3\,(\mod\,4)$.  Then
\begin{itemize}
\item[\rm (a)] $E(\mathcal{G}_{R})=\dfrac{|R^\times|}{2^s}\prod\limits_{i=1}^s\left(\sqrt{|R_i|/m_i}+1\right)$;
\item[\rm (b)] $E(\mathcal{G}_{R_0\times R})=\dfrac{|R_0^\times||R^\times|}{2^{s-1}}\prod\limits_{i=1}^s\left(\sqrt{|R_i|/m_i}+1\right)$.
\end{itemize}
\end{thm}

A graph $G$ with $n$ vertices is called \emph{hyperenergetic} \cite{Gutman99} if $E(G) > 2(n-1)$. By Theorem \ref{Energy13mod4} we obtain the following corollary.

\begin{cor}\label{HyEnerCor1}
Let $R$ and $R_0$ be as in Theorem \ref{Energy13mod4}. Then the following hold:
\begin{itemize}
\item[\rm (a)]  $\mathcal{G}_{R}$ is hyperenergetic except when $R=R_1$ with $|R_1|/m_1=5$ or $R=R_1\times R_2$ with $|R_1|/m_1=|R_2|/m_2=5$;
\item[\rm (b)]  $\mathcal{G}_{R_0\times R}$ is hyperenergetic except when $|R_0|/m_0=3$ and $R=R_1$ with $|R_1|/m_1=5$.
\end{itemize}
\end{cor}

\begin{proof}
(a) By Theorem \ref{Energy13mod4}, $\mathcal{G}_{R}$ is hyperenergetic if and only if
$\dfrac{|R^\times|}{2^s}\prod_{i=1}^s\left(\sqrt{|R_i|/m_i}+1\right)>2(|R|-1)$, which is equivalent to
\begin{equation}\label{Hyperequ1}
\frac{|R^\times|}{|R|}\cdot\dfrac{\prod_{i=1}^s\left(\sqrt{|R_i|/m_i}+1\right)}{2^{s+1}}>1-\frac{1}{|R|}.
\end{equation}

If $s\ge3$, then since $|R_i|/m_i\ge5$, by (\ref{eq:basic}) and $\left(1-\dfrac{1}{|R_i|/m_i}\right)\left(\sqrt{|R_i|/m_i}+1\right)>2$, the left-hand side of (\ref{Hyperequ1}) can be written as
$$
\dfrac{1}{2^{s+1}}\prod_{i=1}^s\left(1-\dfrac{1}{|R_i|/m_i}\right)\left(\sqrt{|R_i|/m_i}+1\right) \ge \dfrac{1}{2}\left(\dfrac{1}{2}\left(1-\dfrac{1}{5}\right)\left(\sqrt{5}+1\right)\right)^3 \approx 1.0844  > 1-\dfrac{1}{|R|}.
$$
Hence $\mathcal{G}_{R}$ is hyperenergetic in this case.

If $s=2$ and $|R_2|/m_2\ge9$, then the left-hand side of (\ref{Hyperequ1}) is greater than or equal to
$\dfrac{1}{8}\left(1-\dfrac{1}{5}\right)\left(\sqrt{5}+1\right)\left(1-\dfrac{1}{9}\right)\left(\sqrt{9}+1\right) \approx 1.1506 > 1-\dfrac{1}{|R|}$.
So $\mathcal{G}_{R}$ is hyperenergetic in this case. If $s = 2$ and $|R_1|/m_1=|R_2|/m_2=5$, then (\ref{Hyperequ1}) is not satisfied and hence $\mathcal{G}_{R}$ is not hyperenergetic.

If $s=1$, then (\ref{Hyperequ1}) is mounted to $\left(1-\dfrac{1}{|R_1|/m_1}\right)\left(\sqrt{|R_1|/m_1}+1\right)>4\left(1-\dfrac{1}{|R_1|}\right)$. This inequality holds if and only if $|R_1|/m_1\ge9$, and in this case $\mathcal{G}_{R}$ is hyperenergetic.

\medskip

(b) By Theorem \ref{Energy13mod4}, $\mathcal{G}_{R_0\times R}$ is hyperenergetic if and only if $\dfrac{|R_0^\times||R^\times|}{2^{s-1}}\prod_{i=1}^s\left(\sqrt{|R_i|/m_i}+1\right)>2(|R_0||R|-1)$, which is equivalent to
\begin{equation}\label{Hyperequ2}
\dfrac{1}{2^{s}}\left(1-\dfrac{1}{|R_0|/m_0}\right)\prod_{i=1}^s\left(1-\dfrac{1}{|R_i|/m_i}\right)\left(\sqrt{|R_i|/m_i}+1\right)>1-\frac{1}{|R_0||R|}.
\end{equation}

If $s\ge2$, then since $|R_0|/m_0\ge3$ and $|R_i|/m_i\ge5$, the left-hand side of (\ref{Hyperequ2}) is greater than or equal to
$\dfrac{1}{4}\left(1-\dfrac{1}{3}\right)\left(\left(1-\dfrac{1}{5}\right)\left(\sqrt{5}+1\right)\right)^2 \approx 1.1171 > 1-\dfrac{1}{|R_0||R|}$.
Thus $\mathcal{G}_{R_0\times R}$ is hyperenergetic in this case.

If $s=1$, then the left-hand side of (\ref{Hyperequ2}) is greater than or equal to
$$
\left\{\begin{array}{ll}
               \dfrac{1}{2}\left(1-\dfrac{1}{7}\right)\left(1-\dfrac{1}{5}\right)\left(\sqrt{5}+1\right)\approx1.1095>1-\dfrac{1}{|R_0||R_1|}, &  \text{if $|R_0|/m_0\ge7$} \\[0.4cm]
                 \dfrac{1}{2}\left(1-\dfrac{1}{3}\right)\left(1-\dfrac{1}{9}\right)\left(\sqrt{9}+1\right)=\dfrac{32}{27}>1-\dfrac{1}{|R_0||R_1|}, & \text{if $|R_1|/m_1\ge9$.}
             \end{array}\right.\\
$$
Thus $\mathcal{G}_{R_0\times R}$ is hyperenergetic when $|R_0|/m_0\ge7$ or $|R_1|/m_1\ge9$. On the other hand, if $|R_0|/m_0=3$ and $|R_1|/m_1=5$, then (\ref{Hyperequ2}) is not satisfied and hence $\mathcal{G}_{R_0\times R}$ is not hyperenergetic.
\qed\end{proof}

In the special case where $R = \mathbb{Z}_n$, Theorems \ref{EnergyLQUCG} and \ref{Energy13mod4} and Corollary \ref{HyEnerCor1} together imply the following result.

\begin{cor}
Let $n=p_1^{\alpha_1}\ldots p_s^{\alpha_s}$ be an integer in canonical factorization such that each $p_i \equiv1\,(\mod\,4)$. Let $p\equiv3\,(\mod\,4)$ be a prime and $\alpha\ge1$ an integer. Then
\begin{itemize}
\item[\rm (a)]  $E(\mathcal {G}_{\mathbb{Z}_{p^\alpha}})=2\varphi(p^\alpha)$;
\item[\rm (b)]  $E(\mathcal {G}_{\mathbb{Z}_n}) =\dfrac{\varphi(n)}{2^s}\prod\limits_{i=1}^s\left(\sqrt{p_i}+1\right)$;
\item[\rm (c)]  $E(\mathcal {G}_{\mathbb{Z}_{np^\alpha}}) =\dfrac{\varphi(n)\varphi(p^\alpha)}{2^{s-1}}\prod\limits_{i=1}^s\left(\sqrt{p_i}+1\right)$;
\item[\rm (d)]  $\mathcal {G}_{\mathbb{Z}_{n}}$ is hyperenergetic except when $n=5^{\alpha_1}$;
\item[\rm (e)]  $\mathcal {G}_{\mathbb{Z}_{np^\alpha}}$ is hyperenergetic except when $np^\alpha=3^\alpha\cdot5^{\alpha_1}$.
\end{itemize}
\end{cor}

\section{Spectral moments of quadratic unitary Cayley graphs}
\label{sec:SMQUCA}

By Theorem \ref{SPECLocalQUCG} we obtain the following result.

\begin{thm}\label{SMLQUCG}
Let $R$ be a local ring with maximal ideal $M$ of order $m$. Then the following hold for any integer $k \ge 1$:
\begin{itemize}
\item[\rm (a)] if $|R|/m\equiv 1\,(\mod\,4)$, then $$s_k(\mathcal{G}_{R})= \dfrac{m^k(|R|/m-1)}{2^{k+1}}\left(2(|R|/m-1)^{k-1}+\left(-1+\sqrt{|R|/m}\right)^k+\left(-1-\sqrt{|R|/m}\right)^k\right);$$
\item[\rm (b)] if $|R|/m\equiv 3\,(\mod\,4)$, then $$s_k(\mathcal{G}_{R})= m^k(|R|/m-1)\left((|R|/m-1)^{k-1}+(-1)^k\right).$$
\end{itemize}
\end{thm}

As a consequence of Lemma \ref{tensorspectrum}, we have $s_k(G\otimes H) = s_k(G)\cdot s_k(H)$. In general, by induction the $k$-th spectral moment of the tensor product of a finite number of graphs is equal to the product of the $k$-th moments of the factor graphs. Thus, by Theorem \ref{SMLQUCG}, one can prove the following result.

\begin{thm}\label{SMs13mod4}
Let $R$ be as in Assumption \ref{as:1} such that $|R_i|/m_i\equiv 1\,(\mod\,4)$ for $1 \le i \le s$, and let $R_0$ be a local ring with maximal ideal $M_0$ of order $m_0$ such that $|R_0|/m_0\equiv 3\,(\mod\,4)$. Then the following hold for any integer $k \ge 1$:
\begin{itemize}
\item[\rm (a)] $s_k(\mathcal{G}_{R})=\prod\limits_{i=1}^s\dfrac{m_i^k(|R_i|/m_i-1)}{2^{k+1}}$

 \hspace{1.6cm}$\times\prod\limits_{i=1}^s\left(2(|R_i|/m_i-1)^{k-1}+\left(-1+\sqrt{|R_i|/m_i}\right)^k+\left(-1-\sqrt{|R_i|/m_i}\right)^k\right)$;
\item[\rm (b)] $s_k(\mathcal{G}_{R_0\times R})=m^k(|R|/m-1)\left((|R|/m-1)^{k-1}+(-1)^k\right)\prod\limits_{i=1}^s\dfrac{m_i^k(|R_i|/m_i-1)}{2^{k+1}}$

 \hspace{2.2cm}$\times\prod\limits_{i=1}^s\left(2(|R_i|/m_i-1)^{k-1}+\left(-1+\sqrt{|R_i|/m_i}\right)^k+\left(-1-\sqrt{|R_i|/m_i}\right)^k\right)$.
\end{itemize}
\end{thm}

It is well known that the $k$-th spectral moment of $G$ is equal to the number of closed walks of length $k$ in $G$. Denote by $n_3(G)$ the number of triangles in a graph $G$. Since $s_3(G)=6n_3(G)$ \cite{kn:Cvetkovic10}, Theorems \ref{SMLQUCG} and \ref{SMs13mod4} imply the following formulae.

\begin{cor}\label{Triangle13mod4}
Let $R$ and $R_0$ be as in Theorem \ref{SMs13mod4}. Then
\begin{itemize}
\item[\rm (a)] $n_3(\mathcal{G}_{R_0})=\dfrac{1}{6}\cdot m_0|R_0||R_0^\times|(|R_0|/m_0-2)$;
\item[\rm (b)] $n_3(\mathcal{G}_{R})=\dfrac{1}{6\cdot8^s}\cdot\prod\limits_{i=1}^s\left(m_i|R_i||R_i^\times|(|R_i|/m_i-5)\right)$;
\item[\rm (c)] $n_3(\mathcal{G}_{R_0\times R})=\dfrac{1}{6\cdot8^s}\cdot m_0|R_0||R_0^\times|(|R_0|/m_0-2)\cdot\prod\limits_{i=1}^s\left(m_i|R_i||R_i^\times|(|R_i|/m_i-5)\right)$.
\end{itemize}
\end{cor}

In the special case where $R = \mathbb{Z}_n$, Theorems \ref{SMLQUCG} and \ref{SMs13mod4} and Corollary \ref{Triangle13mod4} together imply the following result.

\begin{cor}\label{SMCORlol121}
Let $n=p_1^{\alpha_1}\ldots p_s^{\alpha_s}$ be an integer in canonical factorization such that each $p_i \equiv1\,(\mod\,4)$. Let $p\equiv3\,(\mod\,4)$ be a prime and $\alpha\ge1$ an integer. Then the following hold for any integer $k \ge 1$:
\begin{itemize}
\item[\rm (a)]  $s_k(\mathcal {G}_{\mathbb{Z}_{p^\alpha}})= (p-1)p^{k(\alpha-1)}\left((p-1)^{k-1}+(-1)^k\right)$;
\item[\rm (b)]  $s_k(\mathcal {G}_{\mathbb{Z}_{n}}) =\prod\limits_{i=1}^s\left(\dfrac{(p_i-1)p_i^{k(\alpha_i-1)}}{2^{k+1}} \cdot\left(2(p_i-1)^{k-1}+\left(-1+\sqrt{p_i}\right)^k+\left(-1-\sqrt{p_i}\right)^k\right)\right)$;
\item[\rm (c)]  $s_k(\mathcal {G}_{\mathbb{Z}_{np^\alpha}}) =(p-1)p^{k(\alpha-1)}\left((p-1)^{k-1}+(-1)^k\right)$

              \hspace{1.6cm}$\times\prod\limits_{i=1}^s\left(\dfrac{(p_i-1)p_i^{k(\alpha_i-1)}}{2^{k+1}} \cdot\left(2(p_i-1)^{k-1}+\left(-1+\sqrt{p_i}\right)^k+\left(-1-\sqrt{p_i}\right)^k\right)\right)$;
\item[\rm (d)]  $n_3(\mathcal {G}_{\mathbb{Z}_{p^\alpha}})= \dfrac{p^{3\alpha-2}(p-1)(p-2)}{6}$;
\item[\rm (e)]  $n_3(\mathcal {G}_{\mathbb{Z}_{n}}) = \dfrac{1}{6\cdot8^s}\cdot\prod\limits_{i=1}^s\left(p_i^{3\alpha_i-2}(p_i-1)(p_i-5)\right)$;
\item[\rm (f)]  $n_3(\mathcal {G}_{\mathbb{Z}_{np^\alpha}}) =\dfrac{1}{6\cdot8^s}\cdot p^{3\alpha-2}(p-1)(p-2)\cdot\prod\limits_{i=1}^s\left(p_i^{3\alpha_i-2}(p_i-1)(p_i-5)\right)$.
\end{itemize}
\end{cor}

\begin{rem}
{\em The \emph{quadratic residue Cayley graph} of $\mathbb{Z}_n$ \cite{kn:Giudici00,Maheswari09} is defined as the Cayley graph $\Cay(\mathbb{Z}_n,S\cup(-S))$ on the additive group $\mathbb{Z}_n$, where $S$ denotes the set of \emph{quadratic residues} of $\mathbb{Z}_n$. This graph is usually different from ${\cal G}_{\mathbb{Z}_n}$, though the two graphs are identical when $n$ is an odd prime. In \cite{Maheswari09}, the number of triangles of $\Cay(\mathbb{Z}_{p},S\cup(-S))$ was expressed in terms of  the number of `fundamental triangles'. In a more explicit way, Corollary \ref{SMCORlol121} implies that the number of triangles of $\Cay(\mathbb{Z}_{p},S\cup(-S)) = \mathcal{G}_{\mathbb{Z}_{p}}$ is equal to $p(p-1)(p-5)/48$ when $p\equiv1\,(\mod\,4)$ and $p(p-1)(p-2)/6$ when $p\equiv3\,(\mod\,4)$.
}
\end{rem}

\section{Ramanujan quadratic unitary Cayley graphs}

\label{sec:RamaQUCA}

\begin{thm}\label{RamQUCG13mod4}
Let $R$ be as in Assumption \ref{as:1} such that $|R_i|/m_i\equiv 1\,(\mod\,4)$ for $1 \le i \le s$, and let $R_0$ be a local ring with maximal ideal $M_0$ of order $m_0$ such that $|R_0|/m_0\equiv 3\,(\mod\,4)$. Then the following hold:
\begin{itemize}
\item[\rm (a)] $\mathcal{G}_{R_0}$ is Ramanujan if and only if $|R_0|\ge(m_0+2)^2/4$;
\item[\rm (b)] $\mathcal{G}_{R}$ is Ramanujan if and only if $R$ is isomorphic to $\mathbb{F}_5\times\mathbb{F}_5$ or $\mathbb{F}_q$ for a prime power $q\equiv 1\,(\mod\,4)$;
\item[\rm (c)] $\mathcal{G}_{R_0\times R}$ is Ramanujan if and only if $R_{0}\times R$ is isomorphic to $\mathbb{F}_3\times\mathbb{F}_5$, $\mathbb{F}_3\times\mathbb{F}_9$ or $\mathbb{F}_3\times\mathbb{F}_{13}$.
\end{itemize}
\end{thm}
\begin{proof}
(a) By (b) of Theorem \ref{SPECLocalQUCG}, $\mathcal{G}_{R_0}$ is Ramanujan if and only if $m_0\le2\sqrt{|R_0|-m_0-1}$, which is equivalent to $|R_0|\ge(m_0+2)^2/4$.

\medskip

(b) Theorem \ref{Spec1mod4} implies that $\mathcal{G}_{R}$ is Ramanujan if and only if $|\lambda_{A,B}|\le 2\sqrt{|R^\times|/2^s-1}$ for all eigenvalues $\lambda_{A,B}\neq\pm|R^\times|/2^s$. Note that $|\lambda_{A,B}|<|R^\times|/2^s$ is maximized if and only if $\prod_{i\in A}\left(\sqrt{|R_i|/m_i}+1\right)\prod_{j\in B}\left(\sqrt{|R_j|/m_j}-1\right)$ is minimized. Since $|\lambda_{A,B}|\leq|\lambda_{\emptyset,\{1\}}|\neq|R^\times|/2^s$, $\mathcal{G}_{R}$ is Ramanujan if and only if $|\lambda_{\emptyset,\{1\}}|\leq2\sqrt{|R^\times|/2^s-1}$, that is,
\begin{equation}\label{Ram1mod4C1}
  \dfrac{|R^\times|}{2^s\left(\sqrt{|R_1|/m_1}-1\right)}\leq2\sqrt{|R^\times|/2^s-1}.
\end{equation}
Since $2\sqrt{|R^\times|/2^s-1}<2\sqrt{|R^\times|/2^s}$, this condition is not satisfied unless
\begin{equation}\label{Ram1mod4C2}
 |R^\times|/2^s<4\left(\sqrt{|R_1|/m_1}-1\right)^2.
\end{equation}
In particular, if $s\ge4$, then since $\left(\sqrt{|R_1|/m_1}-1\right)^2=|R_1|/m_1-2\sqrt{|R_1|/m_1}+1<|R_1|/m_1-1$, by (\ref{eq:basic}) we have $|R^\times|/2^s\ge\dfrac{1}{2^s}\prod_{i=1}^s\left((|R_{i}|/m_{i})-1\right)$ $\ge4\left((|R_{1}|/m_{1})-1\right)>4\left(\sqrt{|R_1|/m_1}-1\right)^2$, and hence $\mathcal{G}_{R}$ is not Ramanujan. It remains to consider the case where $1\le s\le3$.

\smallskip

\noindent\emph{Case 1:} $s=3$. In view of (\ref{eq:basic}), in this case (\ref{Ram1mod4C2}) is  mounted to
$\dfrac{1}{8}\prod_{i=1}^3m_i\left((|R_{i}|/m_{i})-1\right)<4\left(\sqrt{|R_1|/m_1}-1\right)^2$. If $\prod_{i=1}^3m_i\ge2$ or $|R_{3}|/m_{3}\ge9$, then this condition is not satisfied and hence $\mathcal{G}_{R}$ is not Ramanujan. Now we assume $\prod_{i=1}^3m_i=1$ and $|R_{3}|/m_{3}\le 8$. Then $R_1\cong R_2\cong R_3\cong \mathbb{F}_5$. In this case (\ref{Ram1mod4C1}) is not satisfied and hence $\mathcal{G}_{R}$ is not Ramanujan.

\smallskip

\noindent\emph{Case 2:} $s=2$. In this case (\ref{Ram1mod4C2}) is mounted to
$\dfrac{1}{4}\prod_{i=1}^2m_i\left((|R_{i}|/m_{i})-1\right)<4\left(\sqrt{|R_1|/m_1}-1\right)^2$.
Thus, if $m_1m_2\ge4$ or $|R_{2}|/m_{2}\ge17$, then $\mathcal{G}_{R}$ is not Ramanujan. Assume $m_1m_2\le3$ and $|R_{2}|/m_{2}\le16$. Since $|R_i|/m_i\equiv 1\,(\mod\,4)$ for $i=1,2$, by Lemma \ref{localmaximal}, we have $m_1m_2=1$ or $m_1m_2=3$. It is known \cite{kn:Ganesan64} that $\mathbb{Z}_9$ and $\mathbb{Z}_3[X]/(X^2)$ are the only local rings whose unique maximal ideal has exactly three elements. But their residue fields are $\mathbb{Z}_{3}$, a contradiction to $|R_i|/m_i\equiv 1\,(\mod\,4)$ for $i=1,2$. So $m_1m_2=3$ can not occur. Thus $m_1m_2=1$ and one of the following occurs: (i) $R_1 \cong R_2\cong  \mathbb{F}_5$; (ii) $R_1 \cong R_2\cong  \mathbb{F}_9$; (iii) $R_1 \cong R_2\cong  \mathbb{F}_{13}$; (iv) $R_1 \cong \mathbb{F}_5$ and $R_2\cong  \mathbb{F}_9$; (v) $R_1 \cong \mathbb{F}_5$ and $R_2\cong \mathbb{F}_{13}$; (vi) $R_1 \cong \mathbb{F}_9$ and $R_2\cong \mathbb{F}_{13}$. In case (i), (\ref{Ram1mod4C1}) is satisfied and so $\mathcal{G}_{R}$ is Ramanujan, whilst in (ii)-(vi), (\ref{Ram1mod4C1}) is not satisfied and so $\mathcal{G}_{R}$ is not Ramanujan.

\smallskip

\noindent\emph{Case 3:} $s=1$. In this case (\ref{Ram1mod4C2}) is mounted to
$m_1\left((|R_{1}|/m_{1})-1\right)<8\left(\sqrt{|R_1|/m_1}-1\right)^2$.
Thus, if $m_1\ge8$, then $\mathcal{G}_{R}$ is not Ramanujan. Assume $m_1\le7$. Since $|R_1|/m_1\equiv 1\,(\mod\,4)$, by Lemma \ref{localmaximal}, we have $m_1=1$, $3$, $5$ or $7$. Note that the only finite commutative local rings whose maximal ideal has prime order $p$ are $\mathbb{Z}_{p^2}$ and $\mathbb{Z}_p[X]/(X^2)$. Their residue fields are $\mathbb{Z}_{p}$. Again by $|R_1|/m_1\equiv 1\,(\mod\,4)$, we have $m_1=1$ or $5$. In the former case, $R_1\cong \mathbb{F}_q$, where $q\equiv 1\,(\mod\,4)$ is a prime power. By (\ref{Ram1mod4C1}), $\mathcal{G}_{\mathbb{F}_q}$ is Ramanujan. In the latter case, $R_1\cong \mathbb{Z}_{25}$ or $\mathbb{Z}_5[X]/(X^2)$. Plugging them back into (\ref{Ram1mod4C1}), we obtain that $\mathcal{G}_{R}$ is not Ramanujan in this case.

\medskip

(c) Define $$|\lambda|=\max\left\{|R_0^\times||\lambda_{A,B}|,\,\dfrac{|R_0^\times|}{|R_0|/m_0-1}|\lambda_{A,B}|\right\}.$$ By Theorem \ref{Spec3mod4}, $\mathcal{G}_{R_0\times R}$ is Ramanujan if and only if $|\lambda|\le 2\sqrt{|R_0^\times||R^\times|/2^s-1}$
for $\lambda\neq|R_0^\times||R^\times|/2^s$. Let $\mu_{A,B}=\prod_{i\in A}\left(\sqrt{|R_i|/m_i}+1\right)\prod_{j\in B}\left(\sqrt{|R_j|/m_j}-1\right)$. Note that $|\lambda|<|R_0^\times||R^\times|/2^s$ is maximized if and only if $\min\left\{\mu_{A,B},\,\left(|R_0|/m_0-1\right)\mu_{A,B}\right\}$ is minimized.

\smallskip

\noindent\emph{Case 1:} $\sqrt{|R_1|/m_1}<|R_0|/m_0$. In this case $\mathcal{G}_{R_0\times R}$ is Ramanujan if and only if
\begin{equation}\label{Ram3mod4C1a}
  \dfrac{|R_0^\times||R^\times|}{2^s\left(\sqrt{|R_1|/m_1}-1\right)}\leq2\sqrt{|R_0^\times||R^\times|/2^s-1}.
\end{equation}
Since $2\sqrt{|R_0^\times||R^\times|/2^s-1}<2\sqrt{|R_0^\times||R^\times|/2^s}$, this condition is not satisfied unless
\begin{equation}\label{Ram3mod4C2a}
 |R_0^\times||R^\times|/2^s<4\left(\sqrt{|R_1|/m_1}-1\right)^2.
\end{equation}
In particular, if $s\ge3$, then by (\ref{eq:basic}) we have $|R_0^\times||R^\times|/2^s\ge\dfrac{1}{2^s}(|R_0|/m_0-1)\prod_{i=1}^s\left((|R_{i}|/m_{i})-1\right)$ $\ge4\left((|R_{1}|/m_{1})-1\right)>4\left(\sqrt{|R_1|/m_1}-1\right)^2$, and hence $\mathcal{G}_{R_0\times R}$ is not Ramanujan. It remains to consider the case where $1\le s\le2$.

\smallskip

\noindent\emph{Case 1.1:} $s=2$. In view of (\ref{eq:basic}), in this case (\ref{Ram3mod4C2a}) is  mounted to
$\dfrac{1}{4}\prod_{i=0}^2m_i\left((|R_{i}|/m_{i})-1\right)<4\left(\sqrt{|R_1|/m_1}-1\right)^2$.
Note that if $\prod_{i=0}^2m_i\ge2$ or $|R_{0}|/m_{0}\ge5$ or $|R_{2}|/m_{2}\ge9$, then this condition is not satisfied and hence $\mathcal{G}_{R_0\times R}$ is not Ramanujan. Now we assume $\prod_{i=0}^2m_i=1$, $|R_{0}|/m_{0}\le4$ and $|R_{2}|/m_{2}\le 8$. Then $R_0\cong \mathbb{F}_3$ and $R_1\cong R_2\cong \mathbb{F}_5$. In this case (\ref{Ram3mod4C1a}) is not satisfied and hence $\mathcal{G}_{R_0\times R}$ is not Ramanujan.

\smallskip

\noindent\emph{Case 1.2:} $s=1$. In this case (\ref{Ram3mod4C2a}) is mounted to
$m_0 m_1\left((|R_{0}|/m_{0})-1\right)\left((|R_{1}|/m_{1})-1\right)<8\left(\sqrt{|R_1|/m_1}-1\right)^2$.
Thus, if $m_0m_1\ge4$ or $|R_{0}|/m_{0}\ge9$, then this condition is not satisfied and hence $\mathcal{G}_{R_0\times R}$ is not Ramanujan. Assume $m_0m_1\le3$ and $|R_{0}|/m_{0}\le8$. Since $|R_0|/m_0\equiv 3\,(\mod\,4)$ and $|R_1|/m_1\equiv 1\,(\mod\,4)$, by Lemma \ref{localmaximal} we have $m_0=m_1=1$ or $m_0=3$ and $m_1=1$. In the former case we have $R_{0}\cong \mathbb{F}_3$ or $R_{0}\cong \mathbb{F}_7$. Note that $R_{1}\cong \mathbb{F}_q$ where $q\equiv 1\,(\mod\,4)$ is a prime power. Then $\sqrt{|R_1|/m_1}<|R_0|/m_0$ implies that $R_{0}\times R_1\cong \mathbb{F}_3\times\mathbb{F}_5$, $\mathbb{F}_7\times\mathbb{F}_5$, $\mathbb{F}_7\times\mathbb{F}_9$, $\mathbb{F}_7\times\mathbb{F}_{13}$, $\mathbb{F}_7\times\mathbb{F}_{17}$, $\mathbb{F}_7\times\mathbb{F}_{29}$, $\mathbb{F}_7\times\mathbb{F}_{37}$ or $\mathbb{F}_7\times\mathbb{F}_{41}$. Plugging these back into (\ref{Ram3mod4C1a}), we obtain that only $\mathcal{G}_{\mathbb{F}_3\times\mathbb{F}_5}$ is Ramanujan. In the latter case, $R_{0}\cong \mathbb{Z}_9$ or $R_{0}\cong \mathbb{Z}_3[X]/(X^2)$. Again, $\sqrt{|R_1|/m_1}<|R_0|/m_0$ implies that $R_1\cong \mathbb{F}_5$. By (\ref{Ram3mod4C1a}), $\mathcal{G}_{R_0\times R}$ is not Ramanujan in this case.

\smallskip

\noindent\emph{Case 2:} $\sqrt{|R_1|/m_1}\ge|R_0|/m_0$. In this case $\mathcal{G}_{R_0\times R}$ is Ramanujan if and only if
\begin{equation}\label{Ram3mod4C1b}
  \dfrac{|R_0^\times||R^\times|}{2^s\left(|R_0|/m_0-1\right)}\leq2\sqrt{|R_0^\times||R^\times|/2^s-1}.
\end{equation}
Since $2\sqrt{|R_0^\times||R^\times|/2^s-1}<2\sqrt{|R_0^\times||R^\times|/2^s}$, this condition is not satisfied unless
\begin{equation}\label{Ram3mod4C2b}
 |R_0^\times||R^\times|/2^s<4\left(|R_0|/m_0-1\right)^2.
\end{equation}
In particular, if $s\ge3$, then by (\ref{eq:basic}) we have $|R_0^\times||R^\times|/2^s\ge\dfrac{1}{2^s}(|R_0|/m_0-1)\left(\sqrt{|R_1|/m_1}-1\right)$ $\left(\sqrt{|R_1|/m_1}+1\right)\prod_{i=2}^s\left((|R_{i}|/m_{i})-1\right)$ $>4\left(|R_0|/m_0-1\right)^2$, and hence $\mathcal{G}_{R_0\times R}$ is not Ramanujan. It remains to consider the case where $1\le s\le2$.

\smallskip

\noindent\emph{Case 2.1:} $s=2$. In view of (\ref{eq:basic}), in this case (\ref{Ram3mod4C2b}) is mounted to
$\dfrac{1}{4}\prod_{i=0}^2m_i\left((|R_{i}|/m_{i})-1\right)<4\left(|R_0|/m_0-1\right)^2$.
Note that if $\prod_{i=0}^3m_i\ge2$ or $|R_{2}|/m_{2}\ge9$, then this condition is not satisfied and hence $\mathcal{G}_{R_0\times R}$ is not Ramanujan. Now we assume $\prod_{i=1}^3m_i=1$ and $|R_{2}|/m_{2}\le 8$. Then $R_1\cong R_2\cong \mathbb{F}_5$. Note that $|R_0|/m_0\ge3$, a contradiction to $\sqrt{|R_1|/m_1}\ge|R_0|/m_0$.

\smallskip

\noindent\emph{Case 2.2:} $s=1$. In this case (\ref{Ram3mod4C2b}) is mounted to
$m_0m_1\left(\sqrt{|R_{1}|/m_{1}}+1\right)\left(\sqrt{|R_{1}|/m_{1}}-1\right)<8\left(|R_0|/m_0-1\right)$.
Thus, if $m_0m_1\ge3$ or $|R_{1}|/m_{1}>49$, then this condition is not satisfied and hence $\mathcal{G}_{R_0\times R}$ is not Ramanujan.  Assume $m_0m_1\le2$ and $|R_{1}|/m_{1}\le49$. Since $|R_0|/m_0\equiv 3\,(\mod\,4)$ and $|R_1|/m_1\equiv 1\,(\mod\,4)$, by Lemma \ref{localmaximal} we have $m_0=m_1=1$. Then $R_1\cong \mathbb{F}_5$, $\mathbb{F}_9$, $\mathbb{F}_{13}$, $\mathbb{F}_{17}$, $\mathbb{F}_{29}$, $\mathbb{F}_{37}$, $\mathbb{F}_{41}$ or $\mathbb{F}_{49}$. Then $\sqrt{|R_1|/m_1}\ge|R_0|/m_0$ implies that $R_{0}\times R_1\cong \mathbb{F}_3\times\mathbb{F}_9$, $\mathbb{F}_3\times\mathbb{F}_{13}$, $\mathbb{F}_3\times\mathbb{F}_{17}$, $\mathbb{F}_3\times\mathbb{F}_{29}$, $\mathbb{F}_3\times\mathbb{F}_{37}$, $\mathbb{F}_3\times\mathbb{F}_{41}$, $\mathbb{F}_3\times\mathbb{F}_{49}$ or $\mathbb{F}_7\times\mathbb{F}_{49}$.  Plugging these into (\ref{Ram3mod4C1b}), we obtain that only $\mathcal{G}_{\mathbb{F}_3\times\mathbb{F}_9}$ and $\mathcal{G}_{\mathbb{F}_3\times\mathbb{F}_{13}}$ are Ramanujan.
\qed\end{proof}

Since for a prime power $q\equiv 1\,(\mod\,4)$, $\mathcal{G}_{\mathbb{F}_q}$ is the Paley graph $P(q)$, Theorem \ref{RamQUCG13mod4}(b) may be thought as a strengthening of the known result that $P(q)$ is Ramanujan. In the special case where $R = \mathbb{Z}_n$, Theorem \ref{RamQUCG13mod4} yields the following result.

\begin{cor}\label{SMCORAll121}
Let $n=p_1^{\alpha_1}\ldots p_s^{\alpha_s}$ be an integer in canonical factorization such that each $p_i \equiv1\,(\mod\,4)$. Let $p\equiv3\,(\mod\,4)$ be a prime and $\alpha\ge1$ an integer. Then the following hold:
\begin{itemize}
\item[\rm (a)]  $\mathcal {G}_{\mathbb{Z}_{p^\alpha}}$ is Ramanujan if and only if $p^\alpha=p$ or $p^2$;
\item[\rm (b)]  $\mathcal {G}_{\mathbb{Z}_{n}}$ is Ramanujan if and only if $n=p_1$;
\item[\rm (c)] $\mathcal {G}_{\mathbb{Z}_{np^\alpha}}$ is Ramanujan if and only if $np^\alpha=3\cdot5$ or $np^\alpha=3\cdot13$.
\end{itemize}
\end{cor}

\section{Concluding remarks}
\label{sec:rem}

In this paper we introduced the quadratic unitary Cayley graphs $\mathcal{G}_{R}$ of finite commutative rings $R$ and studied their spectral properties. Such graphs are generalisations of the quadratic unitary Cayley graphs of $\mathbb{Z}_n$ \cite{kn:Beaudrap10} as well as the well known Paley graphs. In the paper we focused on eigenvalues of $\mathcal{G}_{R}$ and related properties. In the case when in the decomposition of $R$ into the direct product of finite local rings $R_i$ with maximal ideals $M_i$ all corresponding finite fields $R_i/M_i$ are of odd orders and at most one of them has order congruent to $3$ modulo $4$, we completely determined the spectra of $\mathcal{G}_{R}$. In the remaining case the eigenvalues of $\mathcal{G}_{R}$ are yet to be determined but different approach would be needed to handle this case.

It would be interesting to study various combinatorial properties of $\mathcal{G}_{R}$ since they give rise to information about $R$. For example, the diameter of $\mathcal{G}_{R}$ is equal to the smallest integer $k \ge 1$ such that every element of $R$ can be expressed as the sum of at most $k$ terms of the form $\pm u^2$ for a unit $u$ of $R$. This observation demonstrates the importance (and the difficulty) of determining the diameter of $\mathcal{G}_{R}$. In the special case when $R = \mathbb{Z}_n$ the diameter as well as the perfectness of $\mathcal{G}_{\mathbb{Z}_n}$ were determined in \cite{kn:Beaudrap10}. In the general case not much is known about the diameter of $\mathcal{G}_{R}$, and progress in this direction (even for some special families of finite commutative rings) will be welcome contributions.

\medskip
\noindent \textbf{Acknowledgements}~~The authors would like to thank Dr. Niel de Beaudrap for his help during the preparation of this paper. X. Liu is supported by MIFRS and MIRS of the University of Melbourne and the Natural Science Foundation of China (No. 11361033). S. Zhou is supported by a Future Fellowship (FT110100629) of the Australian Research Council.

\end{document}